\documentclass[12pt,reqno]{amsart}
\usepackage{amsmath}
\usepackage{amssymb}
\usepackage{amstext}
\usepackage{a4wide}
\usepackage{graphicx}
\allowdisplaybreaks \numberwithin{equation}{section}
\usepackage{color}
\usepackage{cases}

\numberwithin{equation}{section}

\newtheorem{theorem}{Theorem}[section]

\newtheorem{lemma}[theorem]{Lemma}

\theoremstyle{definition}

\newtheorem{definition}[theorem]{Definition}

\theoremstyle{remark}
\newtheorem{remark}[theorem]{Remark}

\begin{document}

\title
[Steady vortex flows of perturbation type]{Steady vortex flows of perturbation type in a planar bounded domain}

 \author{Daomin Cao, Guodong Wang,  Weicheng Zhan}
\address{Institute of Applied Mathematics, Chinese Academy of Sciences, Beijing 100190, and University of Chinese Academy of Sciences, Beijing 100049,  P.R. China}
\email{dmcao@amt.ac.cn}
\address{Institute for Advanced Study in Mathematics, Harbin Institute of Technology, Harbin 150001, P.R. China}
\email{wangguodong14@mails.ucas.edu.cn}
\address{Institute of Applied Mathematics, Chinese Academy of Sciences, Beijing 100190, and University of Chinese Academy of Sciences, Beijing 100049,  P.R. China}
\email{zhanweicheng16@mails.ucas.ac.cn}


\begin{abstract}
In this paper, we investigate steady Euler flows in a two-dimensional bounded domain. By an adaption of the vorticity method, we prove that for any nonconstant harmonic function $q$, which corresponds to a nontrivial irrotational flow, there exists a family of steady Euler flows with small circulation in which the vorticity is continuous and supported in a small neighborhood of the set of maximum points of $q$ near the boundary, and the corresponding stream function satisfies a semilinear elliptic equation with a given profile function.
Moreover, if $q$ has $k$ isolated maximum points $\{\bar{x}_1,\cdot\cdot\cdot,\bar{x}_k\}$ on the boundary, we show that there exists a family of steady Euler flows whose vorticity is continuous and supported in $k$ disjoint regions of small diameter, and each of them is contained in a small neighborhood of $\bar{x}_i$, and in each of these small regions the stream function satisfies a semilinear elliptic equation with a given profile function.

\end{abstract}

\maketitle

\section{Introduction and main results}
In this paper, we study equilibrium states of the incompressible Euler system with fixed boundary condition in a two-dimensional bounded domain. The system of equations is as follows
\begin{align}\label{euler}
\begin{cases}
(\mathbf{v}\cdot\nabla)\mathbf{v}=-\nabla P,&x=(x_1,x_2)\in D,\\
 \nabla\cdot\mathbf{v}=0,\\
 \mathbf v\cdot\mathbf n=g,&x\in\partial D,
\end{cases}
\end{align}
where $D$ is a simply-connected and bounded domain in $\mathbb R^2$ with smooth boundary, $\mathbf{v}=(v_1,v_2)$ is the velocity field, $P$ is the scalar pressure, and $g$ is a given function defined on $\partial D$ satisfying the following compatibility condition
\begin{equation}
\int_{\partial D}gd\sigma=0,
\end{equation}
where $d\sigma$ is the area unit on $\partial D$. If $g\equiv0,$ which is usually called the impermeability boundary condition,
then there is no matter flow through the boundary.

 The scalar vorticity of the fluid is defined by
\[\omega:=\partial_{x_1}v_2-\partial_{x_2}v_1,\]
which describes the relative rotation between two nearby fluid particles. The scalar vorticity is an elementary quantity and plays an important role in the study of 2D Euler flows. See \cite{MB} for example.

We can simplify system \eqref{euler} as single equation of $\omega$ as follows. First we take $``\nabla\times"$ on both sides of the first equation of \eqref{euler} to obtain
\begin{equation}\label{trans}
\nabla\omega\cdot\mathbf v=0,\quad x\in D.
\end{equation}
On the other hand, by the divergence-free condition $\nabla\cdot \mathbf{v}=0$ and Green's theorem, it is not hard to check that there exists some function $\psi$, called the stream function, such that
\[\mathbf{v}=(\partial_{x_2}\psi,-\partial_{x_1}\psi).\]
 For simplicity, throughout this paper we will make constant use of the symbol $\mathbf b^\perp$ to denote the clockwise rotation through $\pi/2$ of some planar vector $\mathbf b$, and we will also write $\nabla^\perp f=(\nabla f)^\perp$ for some function $f$. Thus we have
$\mathbf{v}=\nabla^\perp\psi.$
 It can be easily verified that the stream function satisfies
 \begin{equation}
 \begin{cases}
 -\Delta\psi=\omega,&x\in D,\\
 \nabla^\perp\psi\cdot\mathbf n=g,&x\in\partial D.
 \end{cases}
 \end{equation}
Now let $q$ satisfy
  \begin{equation}
 \begin{cases}
 -\Delta q=0,&x\in D,\\
 \nabla^\perp q\cdot\mathbf n=g,&x\in\partial D.
 \end{cases}
 \end{equation}
Then up to a constant $\psi$ is solved by
\[\psi=q+\mathcal{G}\omega,\,\,\,\mathcal{G}\omega(x):=\int_DG(x,y)\omega(y)dy,\]
where $G$ is the Green's function for $-\Delta$ in $D$ with zero Dirichlet data. Therefore $\mathbf{v}$ can be recovered from $\omega$ as follows
\begin{equation}\label{recover}
\mathbf{v}=\nabla^\perp(q+\mathcal{G}\omega).
\end{equation}
Combining \eqref{trans} and \eqref{recover} we get the following vorticity form of the Euler system
\begin{equation}\label{vor}
\nabla\omega\cdot\nabla^\perp(q+\mathcal{G}\omega)=0.
\end{equation}
To deal with flows whose vorticity is not $C^1$, it is necessary to interpret $\eqref{vor}$ in the following weak sense.

\begin{definition}\label{wsol}
Let $\omega\in L^\infty(D)$. If for any $\phi\in C^\infty_c(D)$ there holds
\begin{equation}\label{wint}
    \int_D\omega\nabla^\perp(\mathcal G\omega+q)\cdot\nabla\phi dx=0,
\end{equation}
then $\omega$ is called a weak solution to the vorticity equation \eqref{vor}.
\end{definition}

It is easy to see that $\omega\equiv0$ is always a solution for any $q$, which is called an irrotational flow. In this case, the velocity field is given by $\mathbf v=\nabla^\perp q$, and the pressure is given by $P=-\frac{1}{2}|\nabla q|^2$. Throughout this paper, we assume that $q$ is nonconstant, which means that this irrotational flow is nontrivial (that is, $\mathbf v\not\equiv0$). We will be focused on the existence of solutions to the Euler system ``near" this nontrivial irrotational flow with nonvanishing vorticity. More precisely, we shall seek steady Euler flow whose vorticity is supported in a finite number of regions of small diameter, and the circulation (that is, the integral of the vorticity) in each of these regions is very small.

Existence of steady Euler flows with nonvanishing vorticity and their stability are very important topics in fluid dynamics, especially when the vorticity is supported in several small regions. Here we recall two types of these flows that have been studied extensively in the literature. The first one is of desingularization type. It consists of constructing solutions in which the vorticity is almost the sum of a finite number of delta measures (also called point vortices). According to the point vortex model (see \cite{L} or \cite{MPu} for a general discussion), the evolution of $k$ point vortices is governed by the following ODE system
\begin{equation}\label{kr}
\frac{dx_i(t)}{dt}=-\kappa_i\nabla_x^\perp h(x_i(t),x_i(t))+ \sum_{j=1,j\neq i}^k\kappa_j\nabla_x^\perp G(x_j(t),x_i(t)),\;i=1,\cdot\cdot\cdot,k.
\end{equation}
where $x_i(t)\in D$ and $\kappa_i\in\mathbb R\setminus\{0\}$ are the location and the strength of the $i$-th point vortex respectively, and $h(\cdot,\cdot)$ is the regular part of the Green's function, that is,
\[h(x,y)=-\frac{1}{2\pi}\ln|x-y|-G(x,y),\,\,x,y\in D.\]
By \eqref{kr}, the equilibrium state of a system of point vortices is a critical point, say $(\tilde{x}_1,\cdot\cdot\cdot,\tilde{x}_k)$, of the following Kirchhoff-Routh function
 \begin{equation*}
{\mathcal W}_k(x_1,\cdot\cdot\cdot,x_k):=- \sum_{i\neq j, 1\leq i,j\leq k}\kappa_i\kappa_jG(x_i,x_j)+ \sum_{i=1}^k\kappa_i^2h(x_i,x_i),\,\,x_i\in D,\,\, x_i\neq x_j \mbox{ if } i\neq j.
\end{equation*}
 Desingularization of point vortices is to construct a family of steady solutions of the Euler system whose vorticity is supported in $k$ disjoint regions of small diameter, each of which is contained in a small neighborhood of $\tilde{x}_i$, and in each of these small regions the circulation of the flow is $\kappa_i$, which is prescribed. During the past decades, desingularization has became a very efficient tool to establish dynamically possible equilibria of 2D Euler flows and analyze their linear and nonlinear stability. See \cite{CLW}\cite{CPY}\cite{CPY2}\cite{CW0}\cite{CW}\cite{SV}\cite{T}\cite{W} and the references therein. The other kind of steady Euler flows is of perturbations type, which is exactly the kind of flows that we are concerned with in this paper. As the name suggests, it is about seeking new Euler flows near a given reference flow, which is usually an irrotational flow. Here we recall several relevant results. In \cite{LYY}, Li--Yan--Yang studied the following elliptic problem
 \begin{equation}\label{lyy}
 \begin{cases}
 -\Delta u=\lambda f(u+q),&x\in D,\\
 u=0,&x\in\partial D,
 \end{cases}
 \end{equation}
 where $\lambda$ is a parameter that is large enough. Note that any smooth solution to \eqref{lyy} produces a steady Euler flow with the vorticity given by $\omega=\lambda f(u+q).$
When the profile function $f$ is $C^1$ and satisfies some additional assumptions, Li--Yan--Yang proved existence of solutions for given $\lambda$ by using the mountain pass lemma and showed that the ``vortex core" $\{x\in D\mid u^\lambda(x)+q(x)>0\}$ shrinks to a global maximum point of $q$ as $\lambda\to+\infty.$ Moreover, they showed that for any given isolated maximum point $x_0$ of $q$, there exists a family of solutions to \eqref{lyy} such that the ``vortex core" shrinks to $x_0$ as $\lambda\to+\infty.$ Unfortunately, with the method in \cite{LYY}, it is hard to construct solutions with separated ``vortex cores".
In \cite{LP}, Li--Peng considered a special case of \eqref{lyy}, that is, $f(s)=s^p_+$ with $p\in(1,+\infty).$ Based on the reduction method, they proved that for $k$ isolated local maximum points of $q$ on $\partial D$, there exists a solution to \eqref{lyy} such that the ``vortex core" has $k$ components that shrinks to these $k$ points respectively as $\lambda\to+\infty$. Note that both results in \cite{LP} and \cite{LYY} are concerned with smooth Euler flows. When the profile function $f$ has a jump continuity, which corresponds to the vortex patch solution, steady Euler flows with separated ``vortex cores" have been established by Cao--Wang--Zhan in \cite{CWZ} via an entirely different method.
 An important feature of the steady Euler flows of perturbation type obtained in \cite{CWZ}\cite{LP}\cite{LYY}, in contrast to the those of desingularization type, is that the circulation vanishes as the parameter changes.

 Our aim in this paper is to generalize the results in \cite{LP} and \cite{LYY} to more general steady Euler flows.
To state our main results, we need some assumptions on the profile function.
 Let $f:\mathbb R\to\mathbb R$ be a function. We make the following assumptions.
 \begin{itemize}
\item[(H1)] $f(s)=0$ for $s\leq 0$;
\item[(H2)]$f$ is continuous and strictly increasing in $[0,+\infty)$;
\item[(H3)]  There exists $\delta_0\in(0,1)$ such that
\[\int_0^sf(r)dr\leq \delta_0f(s)s,\,\,\forall\,s\geq0.\]

\end{itemize}
Note that (H2) and (H3) imply $\lim_{s\to+\infty}f(s)=+\infty$. By using the identity $\int_0^sf(r)dr+\int_0^{f(s)}f^{-1}(r)dr=sf(s)$ for all $s\geq0,$ one can easily check that (H3) is in fact equivalent to
\begin{itemize}
 \item[ (H3)$'$] There exists $\delta_1\in(0,1)$ such that
\[F(s) \ge \delta_1 s f^{-1}(s),\,\,\forall\,s\geq0,\]
where $f^{-1}$ is defined as the inverse function of $f$ in $[0,+\infty)$ and $f^{-1}\equiv0$ in $(-\infty,0]$, and $F(s)=\int_0^sf^{-1}(r)dr$.
\end{itemize}
 Note that $f(s)=s_+^p$ with $p\in(0,+\infty)$ satisfies (H1)--(H3).

Except for the profile function, we also need some assumptions on the $L^\infty$ norm of the vorticity. Let $\Lambda:\mathbb R_+\to\mathbb R_+$ be a function. We make the following assumptions.
 \begin{itemize}
\item[(A1)] $\lim_{s\to0^+}\frac{\Lambda(s)}{s}=+\infty$;
\item[(A2)] There exists some $\gamma_0>0$ such that $\lim_{s\to0^+}\Lambda(s)s^{\gamma_0}=0$.
\end{itemize}

Our first result can be stated as follows.
\begin{theorem}\label{single}
Let $\Lambda:\mathbb R_+\to\mathbb R_+$ be a function satisfying (A1)(A2) and
 $q\in C^2(D)\cap C^1(\overline{D})$ be a harmonic function. Set $\mathcal{S}:=\{x\in \overline{D}\mid q(x)=\max_{ \overline{D}}q\}$. Then there exists $\kappa_0>0$ such that
for any $\kappa\in(0,\kappa_0)$, there exists a weak solution $\omega^\kappa$ to \eqref{vor} having the form
\begin{equation}\label{1-12}
  \omega^\kappa=\Lambda(\kappa) f(\mathcal G\omega^\kappa+q-\mu^\kappa),\,\,\int_D\omega^\kappa dx=\kappa,
\end{equation}
for some $\mu^\kappa\in\mathbb{R}$ depending on $\kappa$.
Furthermore, if $q$ is not a constant, then $\mathcal{S}\subset\partial D$ and $supp(\omega^\kappa)$ approaches $\mathcal{S}$ as $\kappa\rightarrow0^+$, or equivalently, for any $\delta>0$, there exists $\kappa_1>0$, such that for any $\kappa<\kappa_1$, there holds
\begin{equation}\label{1-14}
supp(\omega^\kappa)\subset\mathcal{S}_\delta:=\{x\in D\mid dist(x,\mathcal{S})<\delta\}.
\end{equation}
\end{theorem}

Note that in this paper we will denote by supp$(g)$ the essential support of some measurable function $g$, that is, the complement of the union of all open sets in which $g$ vanishes. See \S 1.5 in \cite{LL} for example.

Our second theorem shows that each finite collection of isolated maximum points of $q$ gives rise to a family of steady Euler flows with multiple ``vortex cores" of small diameter. To state the theorem concisely, we introduce the notion of $\alpha$-uniform corn. Let $k$ be a fixed positive integer and $\alpha$ be a fixed positive number. Define the $\alpha$-uniform cone as follows
\[\mathbb{K}^{\alpha}:=\{\vec{\kappa}\in \mathbb R^k\mid\vec{\kappa}=(\kappa_1,\cdot\cdot\cdot,\kappa_k), \kappa_1,\cdot\cdot\cdot,\kappa_k>0, \max\left\{{\kappa_i}/{\kappa_j}\mid i,j=1,\cdot\cdot\cdot,k\right\}\leq\alpha\}.\]
We also define the norm of $\vec{\kappa}=(\kappa_1,\cdot\cdot\cdot,\kappa_k)\in\mathbb{K}^{\alpha}$ by
\[|\vec{\kappa}|:=\sum_{i=1}^k\kappa_i.\]
It is easy to check that for any $\vec{\kappa}=(\kappa_1,\cdot\cdot\cdot,\kappa_k)\in\mathbb{K}^{\alpha}$, there holds
\[\frac{1}{k\alpha}|\vec{\kappa}|\leq\kappa_i\leq|\vec{\kappa}|,\,\,\forall\,i\in\{1,\cdot\cdot\cdot,k\}.\]
\begin{theorem}\label{multiple}
Let $\Lambda_i:\mathbb R_+\to\mathbb R_+,$ $i=1\cdot\cdot\cdot,k,$ be $k$ functions satisfying
(A1)(A2), $f_i:\mathbb R\to\mathbb R,$ $i=1\cdot\cdot\cdot,k,$ be $k$ functions satisfying (H1)--(H3), and
 $q\in C^2(D)\cap C^1(\overline{D})$ be a harmonic function. Assume that $q$ has $k$ isolated local maximum points $\bar{x}_i,\cdot\cdot\cdot,\bar{x}_k$ on $\partial D$. Then there exists $\delta_0>0$, such that
for any $\vec{\kappa}\in\mathbb K^\alpha$ with $|\vec{\kappa}|<\delta_0$, there exists a weak solution $\omega^{\vec\kappa}$ to \eqref{vor} having the form
\begin{equation}\label{1-12}
  \omega^{\vec\kappa}=\sum_{i=1}^k\omega^{\vec\kappa}_i,\,\,\omega^{\vec\kappa}_i=\Lambda_i(\kappa_i) f_i(\mathcal G\omega^{\vec\kappa}+q-\mu_i^{\vec\kappa})I_{B_{r_0}(\bar{x}_i)},\,\,\int_D\omega^{\vec\kappa}_idx=\kappa_i,
\end{equation}
where $\mu_i^{\vec\kappa}$ is a real number depending on ${\vec\kappa}$, and $r_0$ is a sufficiently small (independent of $\vec\kappa$) positive number such that $B_{r_0}(\bar x_i)\cap B_{r_0}(\bar x_j)=\varnothing$ if $i\neq j$.
Furthermore, for each $i\in\{1,\cdot\cdot\cdot,k\}$, $supp(\omega_i^{\vec\kappa})$ shrinks to $\bar{x}_i$ as $|{\vec\kappa}|\rightarrow0^+$, or equivalently,
\begin{equation}\label{1-14}
supp(\omega_i^{\vec\kappa})\subset B_{o(1)}(\bar{x}_i) \mbox{ as } |\vec\kappa|\to0^+.
\end{equation}
\end{theorem}
\begin{remark}
In Theorem \ref{single}, by the definition of $\alpha$-uniform cone, the circulation of the flow in each $B_{r_0}(\bar x_i)$ vanishes at a uniform rate as $|\vec\kappa|\to0^+$.
\end{remark}

Our strategy of proving Theorem \ref{single} and Theorem \ref{multiple} is basically motivated by \cite{CWZ} and \cite{T}, which is very different from the one used in \cite{LP}\cite{LYY}. In \cite{T}, Turkington studied existence of steady vortex patch solutions of desingularization type and analyzed their limiting behavior as the vorticity strength goes to infinity. More precisely, he considered the following variational problem for the vorticity: maximize the following functional
\begin{equation}
E(\omega):=\frac{1}{2}\int_D\int_DG(x,y)\omega(x)\omega(y)dxdy
\end{equation}
over the admissible class
\[{M}^\lambda:=\{\omega\in L^\infty(D)\mid 0\leq\omega\leq \lambda \mbox{ a.e. in }D,\int_D\omega(x) dx=1\}.\]
Here $\lambda$ is a large positive number representing the vorticity strength. Turkington proved that there exists a maximizer (may be non-unique) and each maximizer $\omega^\lambda$ must be a steady solution to the vorticity equation \eqref{vor} with $q\equiv0$ and has the form
\[\omega^\lambda=\lambda I{_{\{x\in D\mid\mathcal{G}\omega(x)>\mu^\lambda\}}},\]
where $I$ denotes the characteristic function and $\mu^\lambda$ is a real number (the Lagrange multiplier) depending on $\lambda$. Moreover, as $\lambda\to+\infty$, the support of $\omega^\lambda$ shrinks to a global minimum point of the Robin function of the domain. Turkington's method is now called the vorticity method, which was first established by Arnold in 1960s and developed by many authors in the past decades. See \cite{AK}\cite{B1}\cite{B2}\cite{B3}\cite{CW0}\cite{CW}\cite{EM}\cite{ET}\cite{T} and the references therein. Inspired by Turkington's paper, Cao--Wang--Zhan studied existence of steady vortex patch solutions to \eqref{vor} of perturbation type near isolated extreme points of a nonconstant harmonic function. They modified Turkington's approach by considering the maximization (or minimization) of the following functional
\begin{equation}\label{ql}
E_q(\omega)=\frac{1}{2}\int_D\int_DG(x,y)\omega(x)\omega(y)dxdy+\int_Dq(x)\omega(x)dx
\end{equation}
over
\[{N}^\kappa:=\{\omega\in L^\infty(D)\mid 0\leq\omega\leq 1 \mbox{ a.e. in }D,\int_D\omega(x) dx=\kappa\}.\]
where $\kappa$ is a small positive number. They showed that there exists a maximizer (or minimizer) and each maximizer (or minimizer) $\omega^\kappa$ has the form
\[\omega^\kappa=I_{\{x\in D\mid \mathcal{G}\omega^\kappa+q>\mu^\kappa\}}\,\,\,(\mbox{or }\omega^\kappa=I_{\{x\in D\mid \mathcal{G}\omega^\kappa+q<\mu^\kappa\}}),\]
where $\mu^\kappa\in \mathbb R$ depends on $\kappa$. Note that on the right side of \eqref{ql} the first term is quadratic and the second term is linear, so when $\kappa$ is small, the second term will play a dominant role.
Based on this observation, they showed that as $\kappa\to0^+$ the support of $supp(\omega^\kappa)$ shrinks to the set of maximum (or minimum) points of $q$ on the boundary. They also modified $N^\kappa$ to obtain steady vortex patch solutions with multiple ``vortex cores" near any given finite collection of isolated maximum (or minimum) points of $q$.

In this paper, we will use an adaption of the method in \cite{CWZ} to obtain more general steady Euler flows whose vorticity has no discontinuity. Let us first explain how we prove Theorem \ref{single}. Let $f:\mathbb R\to\mathbb R$ satisfy (H1)--(H3) and $\Lambda:\mathbb R_+\to\mathbb R_+$ satisfy (A1)(A2). We consider the maximization of
\begin{equation}\label{mathcale}
\mathcal E(\omega)=\frac{1}{2}\int_D\int_DG(x,y)\omega(x)\omega(y)dxdy+\int_Dq(x)\omega(x)dx -\Lambda(\kappa)\int_DF({\Lambda^{-1}(\kappa)}{\omega(x)})dx
\end{equation}
over
\begin{equation}\label{Mkappa}
\mathcal{M}^\kappa:=\{\omega\in L^\infty(D)\mid 0\leq\omega\leq\Lambda(\kappa), \int_D\omega(x)dx=\kappa\}.
\end{equation}
As compared to \cite{CWZ}, here we add the term $-\Lambda(\kappa)\int_DF({\Lambda^{-1}(\kappa)}{\omega(x)})dx$ in the functional, which is concave and plays a competitive role with the quadratic term.
As will be shown in Section 2, it is not hard to check that there exists a maximizer and every maximizer must be of the form \begin{equation}\label{form0}\omega^\kappa=\Lambda(\kappa)I_{\{x\in D\mid \mathcal{G}\omega^\kappa(x)+q(x)-\mu^\kappa\geq f^{-1}(1)\}}+\Lambda(\kappa)f(\mathcal{G}\omega^\kappa+q-\mu^\kappa)I_{\{x\in D\mid0<\mathcal{G}\omega^\kappa(x)+q(x)-\mu^\kappa<f^{-1}(1)\}}
\end{equation}
for some $\mu^\kappa\in\mathbb R.$ The key ingredient then is to estimate $\mu^\kappa$ as $\kappa\to0^+.$ We will show that $\mu^\kappa=\max_{x\in\overline{D}}q(x)+o(1)$ by a detailed analysis, from which we can easily show that the patch part $\{x\in D\mid \mathcal{G}\omega^\kappa(x)+q(x)-\mu^\kappa\geq f^{-1}(1)\}$ in \eqref{form0} is empty and
obtain the limiting location of supp$(\omega^\kappa)$. In this process, we will make constant use of (A1)(A2) and (H1)--(H3). As to the proof of Theorem \ref{multiple}, we need only to modify the admissible class just as in \cite{CWZ}.

It is also worth mentioning that there are also some perturbation results for 3D Euler flows. See \cite{Al}\cite{TX} and the references therein.

This paper is organized as follows. In Section 2, we prove Theorem \ref{single} by solving the maximization problem of $\mathcal{E}$ over $\mathcal M^\kappa$ and studying the limiting behavior of the maximizer. In Section 3, we modify the admissible class $\mathcal M^\kappa$ by adding some constrains on the support of the vorticity, and then consider a similar maximization problem and give the proof of Theorem \ref{multiple}.

\section{Proof of Theorem \ref{single}}
In this section we prove Theorem \ref{single}. To begin with, as mention in Section 1, we shall study the maximization problem of $\mathcal{E}$ over $\mathcal M^\kappa$.

\subsection{Maximization Problem}
Let  $\mathcal E$ and $\mathcal{M}^\kappa$ be defined by \eqref{mathcale} and \eqref{Mkappa}.
 For simplicity we denote
\[\mathcal F(\omega):=\Lambda(\kappa)\int_DF({\Lambda^{-1}(\kappa)}{\omega(x)})dx.\]
Since $F$ is a convex function, it is apparent that $\mathcal F$ is a convex functional over $\mathcal M^\kappa.$

\begin{lemma}\label{attain}
$\mathcal{E}$ is bounded from above and attains its maximum value over $\mathcal{M}^\kappa$.
\end{lemma}
\begin{proof}
First for any $\omega\in \mathcal{M}^\kappa$, we have
\begin{align*}
 \mathcal E(\omega)&=\frac{1}{2}\int_D\omega(x)\mathcal G\omega(x)dx+\int_Dq(x)\omega(x)dx
  -\Lambda(\kappa)\int_DF({\Lambda^{-1}(\kappa)}{\omega(x)})dx\\
  &=\frac{1}{2}\int_D\int_D G(x,y)\omega(x)\omega(y)dxdy+\int_Dq(x)\omega(x)dx
  -\Lambda(\kappa)\int_DF({\Lambda^{-1}(\kappa)}{\omega(x)})dx\\
  &\leq \frac{\Lambda^2(\kappa)}{2}\int_D\int_D |G(x,y)|dxdy+\Lambda(\kappa)\|q\|_{L^\infty(D)},
\end{align*}
where we used $G(\cdot,\cdot)\in L^1(D\times D)$.
This implies that $\mathcal{E}$ is bounded from above over $\mathcal{M}^\kappa$.

Now let $\{\omega_{j}\}\subset \mathcal{M}^\kappa$ be a maximizing sequence, that is,
\[\lim_{j\to+\infty}\mathcal{E}(\omega_{j}) =\sup_{\omega\in \mathcal{M}^\kappa}\mathcal{E}({\omega}).\]
Since $\mathcal{M}^\kappa$ is bounded, thus a sequentially precompact subset of $L^\infty(D)$ in the weak star topology, we may assume, up to a subsequence, that $\omega_j\to\bar{\omega}$ weakly star in $L^\infty(D)$ as $j\to\infty$ for some $\bar{\omega}\in L^\infty(D)$. Besides, it is not hard to check that $\mathcal{M}^\kappa$ is closed in the weak star topology in $L^\infty(D)$ (see Lemma 3.1 in \cite{CWZ} for example), so $\bar{\omega}\in \mathcal{M}^\kappa.$
Now we show that $\bar{\omega}$ is in fact a maximizer of $\mathcal{E}$ over $\mathcal{M}^\kappa$.
First by elliptic regularity theory we have $\mathcal{G}\omega_{j}\to \mathcal{G}\bar{\omega}$ in $C^1(\overline{D})$, from which we deduce that
\begin{equation}\label{Evar}
\lim_{j\to\infty}\frac{1}{2}\int_D\omega_j(x)\mathcal G\omega_j(x)dx=\frac{1}{2}\int_D\bar\omega(x)\mathcal G\bar\omega(x)dx.
\end{equation}
Second, by the definition of weak star convergence we have
\begin{equation}\label{q}
\lim_{j\to\infty}\int_Dq(x)\omega_j(x)dx=\int_Dq(x)\bar\omega(x)dx.
\end{equation}
Finally, for $\mathcal{F}$ we claim that
\begin{equation}\label{Fvar}
    \liminf_{j\to +\infty} \mathcal{F}(\omega_j)\ge \mathcal{F}(\bar\omega).
\end{equation}
In fact, we can prove \eqref{Fvar} by contradiction. Suppose that \[\liminf_{j\to +\infty} \mathcal{F}(\omega_j)\leq \mathcal{F}(\bar\omega)+2\delta\] for some $\delta>0$. We may take a subsequence, still denoted by $\{\omega_j\}$, such that \[\mathcal{F}(\omega_j)\leq \mathcal{F}(\bar\omega)+\delta\,\,\mbox{ for each $j$}.\]
Since $\omega_j\to\bar{\omega}$ weakly star in $L^\infty(D)$ as $j\to+\infty$, we have $\omega_j\to\bar{\omega}$ weakly in $L^2(D)$ as $j\to\infty$. By Mazur's theorem, we can take a sequence $\{w_n\}$ that converges to $\bar{\omega}$ strongly in $L^2(D)$, where each $w_n$ is made up of convex combinations of the $\omega_j$'s, that is, \[w_n=\sum_{j=1}^{m_n}\theta_{n_j}\omega_j,\,\,\,\sum_{j=1}^{m_n}\theta_{n_j}=1,\,\,\theta_{n_j}\in[0,1].\]
Without loss of generality, we also assume that $w_n$ converges to $\bar{w}$ a.e. in $D$. Then by Lebesgue's dominated convergence theorem we obtain
\begin{equation}\label{aaa}
\lim_{n\to+\infty}\mathcal{F}(w_n)=\mathcal{F}(\bar{\omega}).
\end{equation}
On the other hand,
\begin{equation}\label{convex}
\mathcal{F}(w_n)=\mathcal{F}(\sum_{j=1}^{m_n}\theta_{n_j}\omega_j)\leq
\sum_{j=1}^{m_n}\theta_{n_j}\mathcal{F}(\omega_j)\leq \sum_{j=1}^{m_n}\theta_{n_j}(\mathcal{F}(\bar\omega)+\delta)=\mathcal{F}(\bar\omega)+\delta,
\end{equation}
which is contradiction to \eqref{aaa}. Note that we used the convexity of $\mathcal F$ in the first inequality of \eqref{convex}. Thus we have proved \eqref{Fvar}.
Combining \eqref{Evar}, \eqref{q} and \eqref{Fvar} we get
\begin{equation}\label{leq}
\lim_{j\to+\infty}\mathcal{E}(\omega_{j})\leq \mathcal{E}({\bar\omega}).
\end{equation}
But
\[\lim_{j\to+\infty}\mathcal{E}(\omega_{j}) =\sup_{\omega\in \mathcal{M}^\kappa}\mathcal{E}({\omega})\geq \mathcal{E}({\bar\omega}).\]
Therefore we obtain
\[\mathcal{E}({\bar\omega})=\sup_{\omega\in \mathcal{M}^\kappa}\mathcal{E}({\omega})\]
as desired.
\end{proof}

\begin{lemma}
Let $\omega^\kappa$ be a maximizer of $\mathcal{E}$ over $\mathcal{M}^\kappa$. Then there exists some real number $\mu^\kappa$ depending on $\kappa$ such that
\begin{equation}\label{form}\omega^\kappa=\Lambda(\kappa)I_{\{x\in D\mid \mathcal{G}\omega^\kappa(x)+q(x)-\mu^\kappa\geq f^{-1}(1)\}}+\Lambda(\kappa)f(\mathcal{G}\omega^\kappa+q-\mu^\kappa)I_{\{x\in D\mid 0<\mathcal{G}\omega^\kappa(x)+q(x)-\mu^\kappa<f^{-1}(1)\}}.
\end{equation}

\end{lemma}
\begin{proof}
For any $\omega\in\mathcal M^\kappa,$ define a family of test functions as follows
\begin{equation*}
  \omega_{s}=\omega^{\kappa}+s({\omega}-\omega^{\kappa}),\ \ \ s\in[0,1].
\end{equation*}
Since $\mathcal M^\kappa$ is obviously a convex set, we have $\omega_{s}\in\mathcal M^\kappa$ for any $s\in[0,1]$.
Therefore, by the fact that $\omega^\kappa$ is a maximizer, we have
\begin{equation*}
     0  \ge \frac{d\mathcal{E}(\omega_{s})}{ds}\bigg|_{s=0^+}
       =\int_{D}({\omega}-\omega^{\kappa})\left(\mathcal{G}\omega^{\kappa}+q-f^{-1}(\Lambda^{-1}(\kappa)\omega^\kappa)\right)dx,
 \end{equation*}
that is,
\begin{equation*}
  \int_{D}\omega^{\kappa}\left(\mathcal{G}\omega^{\kappa}+q-f^{-1}(\Lambda^{-1}(\kappa)\omega^\kappa)\right)dx\ge \int_{D}{\omega}\left(\mathcal{G}\omega^{\kappa}+q-f^{-1}(\Lambda^{-1}(\kappa)\omega^\kappa)\right)dx.
\end{equation*}
Since ${\omega}\in \mathcal{M}^\kappa$ is arbitrary,
by an adaptation of the bathtub principle (see \cite{LL}, Theorem 1.14), there exists some real number $\mu^\kappa$ such that
\[\omega^\kappa=\Lambda(\kappa)I_{\{x\in D\mid\mathcal{G}\omega^{\kappa}(x)+q(x)-f^{-1}(\Lambda^{-1}(\kappa)\omega^\kappa(x))>\mu^\kappa\}}
+hI_{\{x\in D\mid \mathcal{G}\omega^{\kappa}(x)+q(x)-f^{-1}(\Lambda^{-1}(\kappa)\omega^\kappa(x))=\mu^\kappa\}},\]
where $h\in L^{1}_{loc}(D)$ satisfies $0\leq h(x)\leq \Lambda(\kappa)$ for a.e. $x\in D$. Since $\omega^\kappa\in\mathcal M^\kappa,$ it is easy to check that
\begin{equation}\label{3-6}
  \begin{split}
    \mathcal{G}\omega^{\kappa}+q-\mu^\kappa &\ge f^{-1}(\Lambda^{-1}(\kappa)\omega^\kappa) \ \ \  \mbox{a.e. on}\  \{x\in D\mid \omega^{\kappa}(x)={\Lambda(\kappa)}\}, \\
     \mathcal{G}\omega^{\kappa}+q-\mu^\kappa &=  f^{-1}(\Lambda^{-1}(\kappa)\omega^\kappa)\ \ \  \mbox{a.e. on}\  \{x\in D\mid0<\omega^{\kappa}(x)<{\Lambda(\kappa)}\}, \\
    \mathcal{G}\omega^{\kappa}+q-\mu^\kappa &\le  f^{-1}(\Lambda^{-1}(\kappa)\omega^\kappa) \ \ \  \mbox{a.e. on}\  \{x\in D\mid\omega^{\kappa}(x)=0\}.
  \end{split}
\end{equation}
 Now the desired form \eqref{form} follows immediately.

\end{proof}

\begin{remark}
Note that one can also consider the minimization of $\mathcal{E}$ over $\mathcal{M}^\kappa.$  In this case, a minimizer indeed exists, but it is a vortex patch, which is exactly the one obtained in \cite{CWZ}. To see this, recall
$\mathcal E=E_q-\mathcal{F}$, where $E_q$ is defined by \eqref{ql}.
By \cite{CWZ}, $E_q$ has a unique minimizer $w^\kappa$ over $\mathcal M^\kappa$ satisfying
$$w^\kappa=\Lambda(\kappa)I_{U^\kappa},\,\,U^\kappa=\{x\in D\mid \mathcal Gw^\kappa(x)+q(x)<\nu^\kappa\},\,\,\Lambda(\kappa)|U^\kappa|=\kappa,$$
where $\nu^\kappa\in\mathbb R$ depends on $\kappa.$ On the other hand, it is also easy to check that for any $\omega\in\mathcal M^\kappa$
\[\mathcal{F}(\omega)=\Lambda(\kappa)\int_DF(\Lambda^{-1}(\kappa)\omega)dx
=\Lambda(\kappa)\int_D\Lambda^{-1}(\kappa)\omega\frac{F(\Lambda^{-1}(\kappa)\omega)}{\Lambda^{-1}(\kappa)\omega}dx
\leq \int_D\omega F(1)dx=\mathcal{F}(w^\kappa).\]
Here we used the convexity of $F$. This means that $w^\kappa$ is a maximizer of $\mathcal F$ over $\mathcal M^\kappa$. So $w^\kappa$ is in fact the unique minimizer of $\mathcal E$ over $\mathcal M^\kappa.$
In other words, the term $\mathcal F$ has no effect in the minimization case.
\end{remark}

\subsection{Limiting behavior}\label{lmt}
In this subsection, we analyze the limiting behavior of $\omega^\kappa$ as $\kappa\to0^+$.
The key ingredient is to derive a suitable estimate for the Lagrange multiplier $\mu^\kappa$.
In the rest of this section, for ease of notations, we will use $o(1)$ to denote various quantities that go to zero as $\kappa\to0^+$, and $o(\kappa)$ to denote $o(1)\kappa.$

\begin{lemma}\label{o1}
$\sup_{\omega\in\mathcal{M}^\kappa}\|\mathcal{G}\omega\|_{W^{1,2}(D)}=o(1),\,\,\sup_{\omega\in\mathcal{M}^\kappa}\|\mathcal G\omega\|_{L^\infty(D)}=o(1).$

\end{lemma}
\begin{proof}
By Sobolev's inequality, it suffices to show that
\[\sup_{\omega\in\mathcal{M}^\kappa}\|\mathcal{G}\omega\|_{W^{2,p}(D)}=o(1)\]
for some $p>1.$ By elliptic regularity theory this reduces to
\begin{equation}\label{reduce}
\sup_{\omega\in\mathcal{M}^\kappa}\|\omega\|_{L^p(D)}=o(1).
\end{equation}
For any $\omega\in\mathcal{M}^\kappa$ we calculate directly to obtain
\[\|\omega\|_{L^p(D)}=\left(\int_D|\omega|^pdx\right)^{1/{p}}\leq |\Lambda(\kappa)|^{1-1/p}\kappa^{1/p}
=\left(\Lambda(\kappa)\kappa^{1/(p-1)}\right)^{1-1/p}.\]
Recalling (A2) and by choosing $p=1+\gamma_0^{-1}$ we get \eqref{reduce}.

\end{proof}

The estimate of the upper bound of $\mu^\kappa$ is straightforward.
\begin{lemma}\label{efqu}
$\mu^\kappa\leq \max_{\overline{D}}q+o(1).$
\end{lemma}
\begin{proof}
It is obvious that $\{x\in D\mid \mathcal{G}\omega^\kappa(x)+q(x)-\mu^\kappa>0\}$ is not empty, so we have
\[\mu^\kappa\leq\|\mathcal G\omega^\kappa\|_{L^\infty(D)}+\max_{\overline{D}}q.\]
Combining Lemma \ref{o1} we complete the proof.

\end{proof}

The estimate of the lower bound of $\mu^\kappa$ is a little involved. We begin with the following lemma.
\begin{lemma}\label{lowerb}
$\mathcal{E}(\omega^\kappa)\geq \kappa \max_{\overline{D}}q+o(\kappa).$
\end{lemma}

\begin{proof}
The basic idea is to choose a suitable test function. Let $x_0\in\partial D$ be a maximum point of $q$ on $\overline{D}$. Since $\partial D$ is smooth, $D$ satisfies the interior sphere condition at $x_0\in \partial D$. Therefore, for $\kappa$ sufficiently small, we can choose a disc $B_\varepsilon(x^\kappa)\subset D$ with $|x^\kappa-x_0|=\varepsilon$, where $\varepsilon$ satisfies $\pi\varepsilon^2=\sqrt{\kappa/\Lambda(\kappa)}$. Now we define a test function to be $\upsilon^\kappa:=\sqrt{\kappa \Lambda(\kappa)}I_{B_\varepsilon(x^\kappa)}.$ It is obvious that $\upsilon^\kappa\in \mathcal{M}^\kappa$ for small $\kappa$. Since $\omega^\kappa$ is a maximizer, we have
\begin{equation}
  \mathcal E(\omega^\kappa)\geq \mathcal E(\upsilon^\kappa).
  \end{equation}
On the other hand,
\begin{equation}\label{three}
\begin{split}
\mathcal E(\upsilon^\kappa)=\frac{1}{2}\int_{B_\varepsilon(x^\kappa)}G\upsilon^\kappa(x)dx+\int_{B_\varepsilon(x^\kappa)}q(x)dx
-\Lambda(\kappa)\int_{B_\varepsilon(x^\kappa)}F\left(\sqrt{{\kappa}/{\Lambda(\kappa)}}\right)dx.
  \end{split}
\end{equation}
For the first two terms in \eqref{three}, it is easy to check by a direct calculation that
\begin{equation}
\left|\int_{B_\varepsilon(x^\kappa)}G\upsilon^\kappa(x)dx\right|\leq \kappa \|G\upsilon^\kappa\|_{L^\infty(D)}=o(\kappa),
\end{equation}
\begin{equation}\label{efqq}
\int_{B_\varepsilon(x^\kappa)}q(x)dx\geq\kappa q(x_0)-\int_{B_\varepsilon(x^\kappa)}|q(x)-q(x_0)|dx\geq
\kappa q(x_0)-\kappa\varepsilon\|\nabla q\|_{L^\infty(D)}=\kappa q(x_0)+o(\kappa),
\end{equation}
where we used $q\in C^1(\overline{D})$, thus $\|\nabla q\|_{L^\infty(D)}<+\infty$, in \eqref{efqq}.
For the last term in \eqref{three}, since $F(s)\le s f^{-1}(s)$ and $\lim_{s\to0^+}\Lambda(s)/s=+\infty$ (recall assumption (A1)), we have
\begin{equation}
\left|\Lambda(\kappa)\int_{B_\varepsilon(x^\kappa)}F\left(\sqrt{{\kappa}/{\Lambda(\kappa)}}\right)dx\right|\le \kappa f^{-1}\left(\sqrt{{\kappa}/{\Lambda(\kappa)}}\right)=o(\kappa).
\end{equation}
Now the desired result clearly follows.

\end{proof}

\begin{lemma}\label{efF}
$\mathcal F({\omega^\kappa})= o(\kappa).$
\end{lemma}
 \begin{proof}
  In view of Lemma \ref{o1} and Lemma \ref{lowerb}, we calculate directly as follows
    \begin{align}\label{Ff}
    \mathcal F({\omega^\kappa})&\leq \frac{1}{2}\int_D\omega^\kappa\mathcal G\omega^\kappa dx+\int_Dq\omega^\kappa dx
    -\kappa \max_{\overline{D}}q+o(\kappa)\\
    &=\frac{1}{2}\int_D\omega^\kappa\mathcal G\omega^\kappa dx+\int_D(q-\max_{\overline{D}}q)\omega^\kappa dx
    +o(\kappa)\\
    &\leq  \frac{1}{2}\kappa\|\mathcal G\omega^\kappa\|_{L^\infty(D)}+o(\kappa)\\
    &=o(\kappa).
    \end{align}
 Thus the proof is completed.
\end{proof}

Now we are in a position to derive the desired estimate for $\mu^\kappa$.
\begin{lemma}\label{mu}
$\mu^\kappa=\max_{\overline{D}}q+o(1).$
\end{lemma}

\begin{proof}
Since we have proved $\mu^\kappa\leq \max_{\overline{D}} q+o(1)$ in Lemma \ref{efqu}, it suffices to show that
\[\mu^\kappa\geq \max_{\overline{D}} q+o(1).\]
We write
\begin{equation}\label{write}
\begin{split}
\mathcal{E}(\omega^\kappa)
=&\frac{1}{2}\int_D\omega^\kappa\mathcal{G}\omega^\kappa dx+\int_Dq\omega^\kappa dx
-\mathcal F({\omega^\kappa})\\
=&-\frac{1}{2}\int_D\omega^\kappa\mathcal{G}\omega^\kappa dx+\int_D\omega^\kappa(\mathcal{G}\omega^\kappa+q-\mu^\kappa)dx
-\mathcal F({\omega^\kappa})+\kappa\mu^\kappa\\
=&\kappa\mu^\kappa+\int_D\omega^\kappa(\mathcal{G}\omega^\kappa+q-\mu^\kappa)dx+o(\kappa),
\end{split}
\end{equation}
where in the third equality we used Lemma \ref{efF}. Taking into account \eqref{write} and Lemma \ref{lowerb}, it remains to show
\[\left|\int_D\omega^\kappa(\mathcal{G}\omega^\kappa+q-\mu^\kappa)dx\right|=o(\kappa).\]
To this end, we calculate as follows
\begin{align*}
&\int_D\omega^\kappa(\mathcal{G}\omega^\kappa+q-\mu^\kappa)dx\\
=&\int_{\{x\in D\mid 0<\omega^\kappa(x)<\Lambda(\kappa)\}}\omega^\kappa(\mathcal{G}\omega^\kappa+q-\mu^\kappa)dx
+\int_{\{x\in D\mid\omega^\kappa(x)=\Lambda(\kappa)\}}\omega^\kappa(\mathcal{G}\omega^\kappa+q-\mu^\kappa)dx\\
=&\int_{\{x\in D\mid 0<\omega^\kappa(x)<\Lambda(\kappa)\}}\omega^\kappa f^{-1}(\Lambda^{-1}(\kappa)\omega^\kappa)dx
+\int_{\{x\in D\mid\omega^\kappa(x)=\Lambda(\kappa)\}}\omega^\kappa(\mathcal{G}\omega^\kappa+q-\mu^\kappa)dx\\
=&\int_D\omega^\kappa f^{-1}(\Lambda^{-1}(\kappa)\omega^\kappa)dx
+\int_{\{x\in D\mid\omega^\kappa(x)=\Lambda(\kappa)\}}\omega^\kappa(\mathcal{G}\omega^\kappa+q-\mu^\kappa-f^{-1}(1))dx\\
=&\int_D\omega^\kappa f^{-1}(\Lambda^{-1}(\kappa)\omega^\kappa)dx
+\int_D\omega^\kappa(\mathcal{G}\omega^\kappa+q-\mu^\kappa-f^{-1}(1))_+dx.
\end{align*}
For the first term, recalling (H3)$'$, that is,  $\delta_1sf^{-1}(s)\leq F(s)\leq sf^{-1}(s),\,\forall\,s\geq0$ for some $\delta_1\in(0,1)$, we obtain
\begin{align*}
\int_D\omega^\kappa f^{-1}(\Lambda^{-1}(\kappa)\omega^\kappa)dx&=\Lambda(\kappa)\int_D\Lambda^{-1}(\kappa)\omega^\kappa f^{-1}(\Lambda^{-1}(\kappa)\omega^\kappa)dx\\
&\leq \frac{\Lambda(\kappa)}{\delta_1}\int_DF(\Lambda^{-1}(\kappa)\omega^\kappa)dx\\
&=o(\kappa).
\end{align*}
Therefore, to finish the proof, it is enough to verify
\begin{align}\label{want}
\int_D\omega^\kappa(\mathcal{G}\omega^\kappa+q-\mu^\kappa-f^{-1}(1))_+dx=o(\kappa).
\end{align}
Denote $\zeta^\kappa=\mathcal{G}\omega^\kappa+q-\mu^\kappa-f^{-1}(1)$ and $A^\kappa=\{x\in D\mid \omega^\kappa(x)=\Lambda(\kappa)\}.$
By H\"older's inequality and the Sobolev imbedding theorem we have
\begin{equation}\label{ccc}
\begin{split}
&\int_D\omega^\kappa(\mathcal{G}\omega^\kappa+q-\mu^\kappa-f^{-1}(1))_+dx\\
=&\Lambda(\kappa)\int_{A^\kappa}\zeta^\kappa_+dx\\
\leq&\Lambda(\kappa)|A^\kappa|^{1/2}\left(\int_D|\zeta^\kappa_+ |^2dx\right)^{1/2}\\
\leq&C\Lambda(\kappa)|A^\kappa|^{1/2}\int_D\left(|\zeta^\kappa_+ |+|\nabla\zeta^\kappa_+|\right)dx\\
=&C\Lambda(\kappa)|A^\kappa|^{1/2}\int_D|\zeta^\kappa_+ |dx
+C\Lambda(\kappa)|A^\kappa|^{1/2}\int_D|\nabla\zeta^\kappa_+|dx.
\end{split}
\end{equation}
Here and hereafter we use $C$ to denote various quantities not depending on $\kappa$. Since $|A^\kappa|\to0$ as $\kappa\to0$ (by assumption (A1)), we obtain from \eqref{ccc}
\begin{align*}
\int_D\omega^\kappa(\mathcal{G}\omega^\kappa+q-\mu^\kappa-f^{-1}(1))_+dx=\Lambda(\kappa)\int_D|\zeta^\kappa_+ |dx
\leq C\Lambda(\kappa)|A^\kappa|^{1/2}\int_D|\nabla\zeta^\kappa_+|dx.
\end{align*}
But
\begin{align*}
\Lambda(\kappa)|A^\kappa|^{1/2}\int_D|\nabla\zeta^\kappa_+|dx
&\leq \Lambda(\kappa)|A^\kappa|\left(\int_D|\nabla\zeta^\kappa_+|^2dx\right)^{1/2}\\
&\leq C\kappa\left(\int_{A^\kappa}|\nabla q|^2+|\nabla G\omega^\kappa|^2dx\right)^{1/2}\\
&\leq o(\kappa),
\end{align*}
from which \eqref{want} obviously follows.

\end{proof}

\begin{lemma}
$\{x\in D\mid\mathcal G\omega^\kappa(x)+q(x)-\mu^\kappa\geq f^{-1}(1)\}=\varnothing$ if $\kappa$ is sufficiently small. As a consequence, $\omega^\kappa$ has the following form
\[\omega^\kappa=\Lambda(\kappa) f(G\omega^\kappa+q-\mu^\kappa)\]
if $\kappa$ is sufficiently small.
\end{lemma}
\begin{proof}
According to Lemma \ref{mu}, we have $\mathcal G\omega^\kappa+q-\mu^\kappa\leq o(1)$ as $\kappa\to0^+$. But $f^{-1}(1)$ is a positive number not depending on $\kappa$, so $\{x\in D\mid\mathcal G\omega^\kappa(x)+q(x)-\mu^\kappa\geq f^{-1}(1)\}$ must be empty if $\kappa$ is sufficiently small.

\end{proof}

\begin{lemma}\label{shrinkk}
$supp(\omega^\kappa)$ shrinks to $\mathcal S$ as $\kappa\to0$, that is, for any $\delta>0$, there exists $\kappa_1>0,$ such that for any $\kappa<\kappa_1,$ there holds $supp(\omega^\kappa)\subset \mathcal S_\delta.$
\end{lemma}
\begin{proof}
Denote $B^\kappa=\{x\in D\mid \mathcal G\omega^\kappa(x)+q(x)-\mu^\kappa>0\}$. It suffices to show that $B^\kappa$ shrinks to $\mathcal S$ as $\kappa\to0^+$. We prove this by contradiction. Suppose that there exist $\delta_0>0,$ $\{\kappa_j\}\subset\mathbb R_+$, $\{x_j\}\subset B^{\kappa_j}$, $j=1,\cdot\cdot\cdot$, such that $\kappa_j\to 0$ as $j\to+\infty$ but $x_j\notin \mathcal S_{\delta_0}$. Then by continuity of $q$ we have $\sup_jq(x_j)<\max_{\overline{D}}q$, which contradicts Lemma \ref{mu}.

\end{proof}
\subsection{Proof of Theorem \ref{single}}
Now we give the proof of Theorem \ref{single} for completeness.
\begin{proof}[Proof of Theorem \ref{single}]
We need only to show that $\omega^\kappa$ is a steady solution to the vorticity equation, since other assertions of Theorem \ref{single} have been verified in the last subsection.

For any $\phi\in C_c^\infty(D)$, we define a family of transformations $\{\Phi_s\}_{s\in\mathbb R}$ from $D$ to $D$ by solving the following ODE
\begin{equation}\label{ode}
\begin{cases}
  \frac{d\Phi_s(x)}{ds}=-\nabla^\perp\phi(\Phi_s(x)), & s\in\mathbb R \\
  \Phi_0(x)=x.
\end{cases}
\end{equation}
Since $\nabla^\perp\phi$ is a smooth vector field with compact support, $\eqref{ode}$ has a global solution for any $x\in D$. It is also easy to check that $\Phi_s$ is area-preserving, that is, for any measurable set $A\subset D$, we have $|\{\Phi_s(x)\mid x\in A\}|=|A|$ for any $s\in\mathbb R$. Set $\omega_s(x):=\omega^\kappa(\Phi_s(x)).$
Then $\omega_s\in \mathcal{M}^\kappa$ for any $s\in\mathbb R$. This implies
\begin{equation}\label{var0}
\frac{d\mathcal E(\omega_s)}{ds}\bigg|_{s=0}=0.
\end{equation}
On the other hand, we can expand $\mathcal E(\omega_s)$ at $s=0$ to obtain
\[\begin{split}
\mathcal E(\omega_s)
=\mathcal E(\omega^\kappa)+s\int_D\omega^\kappa\nabla^\perp(\mathcal G\omega^\kappa+q)\cdot\nabla\phi dx+o(s),
\end{split}\]
which together with \eqref{var0} gives
\[\int_D\omega^\kappa\nabla^\perp(\mathcal G\omega^\kappa+q)\cdot\nabla\phi dx=0.\]
Thus the proof is completed.

\end{proof}
\begin{remark}
From the proof, it is easy to see that we actually only use the value of $f$ in a small neighborhood of 0, therefore it is sufficient that $f$ only satisfies (H1)--(H3) in $[0,\tau_0]$ for some $\tau_0>0$ rather than in $[0,+\infty)$.
\end{remark}

\section{Proof of Theorem \ref{multiple}}
In this section we give the proof of Theorem \ref{multiple}, which is analogous to that of Theorem \ref{single}.

\subsection{Minimization Problem}

For $\vec\kappa\in\mathbb K^\alpha$ with $|\vec{\kappa}|$ sufficiently small, define
\begin{equation}
\mathcal{N}^{\vec\kappa}:=\{\omega\in L^\infty(D)\mid \omega=\sum_{i=1}^k\omega_i,\,supp(\omega_i)\subset B_{r_0}(\bar{x}_i),\,0\leq\omega_i\leq\Lambda_i(\kappa_i), \int_D\omega_i(x)dx=\kappa_i\}.
\end{equation}
Note that $\mathcal{N}^\kappa$ is not empty if $|\vec\kappa|$ is sufficiently small. We consider the maximization of the following functional over $\mathcal{N}^\kappa$
\begin{equation}
\mathcal P(\omega):=\frac{1}{2}\int_D\omega\mathcal G\omega dx+\int_Dq\omega dx
  -\sum_{i=1}^k\int_D\Lambda_i(\kappa_i)F_i(\Lambda_i^{-1}(\kappa_i)\omega I_{B_{r_0}}(\bar{x}_i))dx,\,\,\omega\in \mathcal{N}^\kappa.
\end{equation}

\begin{lemma}
$\mathcal{P}$ is bounded from above attains its maximum value over $\mathcal{N}^{\vec\kappa}$.
\end{lemma}
\begin{proof}
  An argument similar to the one used in Lemma \ref{attain} shows that $\mathcal{P}$ is bounded from above over $\mathcal{N}^{\vec\kappa}$.

  Next we show that  $\mathcal{P}$ attains its maximum value over $\mathcal{N}^{\vec\kappa}$. Let $\{\omega_{j}\}\subset \mathcal{N}^{\vec\kappa}$ be a maximizing sequence, that is,
\[\lim_{j\to+\infty}\mathcal{P}(\omega_{j}) =\sup_{\omega\in \mathcal{N}^{\vec\kappa}}\mathcal{P}({\omega}).\]
We may assume, up to a subsequence, that $\omega_j\to\bar{\omega}$ weakly star in $L^\infty(D)$ as $j\to\infty$ for some $\bar{\omega}\in L^\infty(D)$. Besides, it is also easy to check that $\mathcal{N}^{\vec\kappa}$ is closed in the weak star topology in $L^\infty(D)$, therefore $\bar{\omega}\in \mathcal{N}^{\vec\kappa}.$
To finish the proof, it is sufficient to show
\begin{equation}\label{leqqq}
\mathcal{P}({\bar\omega})\geq\sup_{\omega\in \mathcal{N}^{\vec\kappa}}\mathcal{P}({\omega}),
\end{equation}
or equivalently,
\begin{equation}\label{leqqqq}
\mathcal{P}({\bar\omega})\ge\lim_{j\to+\infty}\mathcal{P}(\omega_{j}).
\end{equation}
To prove \eqref{leqqqq}, first as in Lemma \ref{attain} we have
\begin{equation}\label{gg}
\lim_{j\to\infty}\frac{1}{2}\int_D\omega_j(x)\mathcal G\omega_j(x)dx=\frac{1}{2}\int_D\bar\omega(x)\mathcal G\bar\omega(x)dx.
\end{equation}
\begin{equation}\label{qq}
\lim_{j\to\infty}\int_Dq(x)\omega_j(x)dx=\int_Dq(x)\bar\omega(x)dx.
\end{equation}
For the last term, we observe that $\omega_j I_{B_{r_0}}(\bar{x}_i)$ converges to $\bar{\omega} I_{B_{r_0}}(\bar{x}_i)$ in the weak star topology of $L^\infty(D)$ as $j\to+\infty$ for each $i\in\{1,\cdot\cdot\cdot,k\}.$ So proceeding as in Lemma \ref{attain}, we obtain
\begin{equation}\label{FFF}
\int_D\Lambda_i(\kappa_i)F_i(\Lambda_i^{-1}(\kappa_i)\bar{\omega} I_{B_{r_0}}(\bar{x}_i))dx
\leq\liminf_{j\to+\infty}\int_D\Lambda_i(\kappa_i)F_i(\Lambda_i^{-1}(\kappa_i)\omega_j I_{B_{r_0}}(\bar{x}_i))dx
\end{equation}
for each $i\in\{1,\cdot\cdot\cdot,k\}.$
Combining \eqref{gg}, \eqref{qq} and \eqref{FFF} we get \eqref{leqqqq}. Thus the proof is completed.

\end{proof}

\begin{lemma}
Let $\omega^{\vec \kappa}$ be a maximizer of $\mathcal P$ over $\mathcal N^{\vec\kappa}$. Set $\omega^{\vec{\kappa}}_i=\omega^{\vec{\kappa}}I_{B_{r_0}}(\bar{x}_i)), i=1,\cdot\cdot\cdot,k$. Then for each $i$ there exists a real number $\mu^{\vec\kappa}_i$ such that
\[\int_D\omega^{\vec\kappa}_idx=\kappa_i,\]
\[\omega^{\vec\kappa}_i=\Lambda_i(\kappa_i)I_{\{x\in B_{r_0}(\bar x_i)\mid\phi^{\vec{\kappa}}_i(x)\geq f_i^{-1}(1)\}}+\Lambda_i(\kappa_i)f_i(\phi^{\vec{\kappa}}_i)I_{\{x\in B_{r_0}(\bar x_i)\mid 0<\phi^{\vec{\kappa}}_i(x)<f_i^{-1}(1)\}},\]
where
\[\phi^{\vec{\kappa}}_i:=\mathcal{G}\omega^{\vec\kappa}+q-\mu_i^{\vec\kappa}.\]
\end{lemma}
\begin{proof}
Let $i$ be a fixed index. Define $w=\sum_{j=1,j\neq i}^k\omega^{\vec\kappa}_j+w_i,$ where $w_i$ satisfies
\begin{equation}\label{wi}
supp(w_i)\subset B_{r_0}(\bar{x}_i),\,\,0\leq w_i\leq \Lambda_i(\kappa_i)\mbox{ a.e. in }D,\,\,\int_Dw_i dx=\kappa_i.
\end{equation}
Obviously $w$ belongs to $\mathcal N^{\vec\kappa}$. It is easy to see that $\mathcal N^{\vec\kappa}$ is a convex set, so $w_s:=\omega^{\vec\kappa}+s(w-\omega^{\vec\kappa})\in \mathcal N^{\vec\kappa}$ for each $s\in[0,1]$. Taking into account the fact that
$\omega^{\vec\kappa}$ is a maximizer of $\mathcal P$ over $\mathcal N^{\vec\kappa}$, we obtain
\begin{align*}
0&\geq\frac{\mathcal P(w_s)}{ds}\bigg|_{s=0^+}=\int_D(\mathcal G\omega^{\vec\kappa}+q+f_i^{-1}(\Lambda_i(\kappa_i)\omega_i^{\vec\kappa}))(w_i-\omega_i^{\vec\kappa})dx.
\end{align*}
That is,
\[\int_D(\mathcal G\omega^{\vec\kappa}+q+f_i^{-1}(\Lambda_i(\kappa_i)\omega_i^{\vec\kappa}))\omega_i^{\vec\kappa}dx
\geq\int_D(\mathcal G\omega^{\vec\kappa}+q+f_i^{-1}(\Lambda_i(\kappa_i)\omega_i^{\vec\kappa}))w_idx\]
for arbitrary $w_i$ satisfying \eqref{wi}. By an adaption of the bathtub principle, we deduce that
there exists some real number $\mu_i^{\vec\kappa}$ such that
\[\omega_i^{\vec\kappa}=\Lambda_i(\kappa_i)I_{\{x\in B_{r_0}(\bar{x}_i)\mid\mathcal G\omega^{\vec\kappa}+q+f_i^{-1}(\Lambda_i(\kappa_i)\omega_i^{\vec\kappa})>\mu_i^{\vec\kappa}\}}
+hI_{\{x\in B_{r_0}(\bar{x}_i)\mid \mathcal G\omega^{\vec\kappa}+q+f_i^{-1}(\Lambda_i(\kappa_i)\omega_i^{\vec\kappa})=\mu_i^{\vec\kappa}\}},\]
where $h\in L^{1}_{loc}(B_{r_0}(\bar{x}_i))$ satisfies $0\leq h(x)\leq \Lambda_i(\kappa_i)$ a.e. $x\in B_{r_0}(\bar{x}_i)$. This implies
\begin{equation*}
  \begin{split}
    \mathcal{G}\omega^{\vec\kappa}+q-\mu_i^{\vec\kappa} &\ge f_i^{-1}(\Lambda_i^{-1}(\kappa_i)\omega_i^{\vec\kappa}) \ \ \  \mbox{a.e. on}\  \{x\in B_{r_0}(\bar{x}_i)\mid \omega_i^{\vec\kappa}(x)={\Lambda_i(\kappa_i)}\}, \\
     \mathcal{G}\omega^{\vec\kappa}+q-\mu_i^{\vec\kappa} &=f_i^{-1}(\Lambda_i^{-1}(\kappa_i)\omega_i^{\vec\kappa})\ \ \  \mbox{a.e. on}\  \{x\in B_{r_0}(\bar{x}_i)\mid0<\omega_i^{\vec\kappa}(x)<{\Lambda_i(\kappa_i)}\}, \\
   \mathcal{G}\omega^{\vec\kappa}+q-\mu_i^{\vec\kappa} &\le f_i^{-1}(\Lambda_i^{-1}(\kappa_i)\omega_i^{\vec\kappa})\ \ \  \mbox{a.e. on}\  \{x\in B_{r_0}(\bar{x}_i)\mid\omega_i^{\vec\kappa}(x)=0\}.
  \end{split}
\end{equation*}
 Now the desired result follows immediately.

\end{proof}

\subsection{Limiting behavior}
Now we study the limiting behavior of each $\omega^{\vec\kappa}_i$ as $|\vec{\kappa}|\to0^+.$ As in Section 2, the key point is to estimate each $\mu^{\vec\kappa}_i.$ In the rest of this section, we will use $o(1)$ to denote various quantities that go to zero as $|\vec\kappa|\to0^+$, and $o(|\vec\kappa|)$ to denote $o(1)|\vec\kappa|.$
\begin{lemma}\label{oo1}
  $\sup_{\omega\in\mathcal{N}^{\vec\kappa}}\|\mathcal{G}\omega\|_{W^{1,2}(D)}=o(1),\,\,\sup_{\omega\in\mathcal{N}^{\vec\kappa}}\|\mathcal G\omega\|_{L^\infty(D)}=o(1).$

\end{lemma}
\begin{proof}
Recall that for each $\vec\kappa=(\kappa_1,\cdot\cdot\cdot,\kappa_k)\in\mathbb K^\alpha$, there holds $(k\alpha)^{-1}|\vec\kappa|\leq\kappa_i\leq|\vec\kappa|$ for each $i\in\{1,\cdot\cdot\cdot,k\}$. Then the remainder of the proof is analogous to that of Lemma \ref{o1}.

\end{proof}

\begin{lemma}\label{qleq}
$\mu_i^{\vec{\kappa}}\leq q(\bar{x}_i)+o(1).$
\end{lemma}
\begin{proof}
  For each $i$ and any $\vec\kappa,$ since $\{x\in B_{r_0}(\bar{x}_i)\mid\mathcal{G}\omega^{\vec\kappa}(x)+q(x)-\mu_i^{\vec\kappa}>0\}$ is nonempty, so $$\mu_i^{\vec\kappa}\leq \|\mathcal{G}\omega^{\vec\kappa}\|_{L^\infty(B_{r_0}(\bar{x}_i))}+\|q\|_{L^\infty(B_{r_0}(\bar{x}_i))}=o(1)+q(\bar x_i).$$
\end{proof}

\begin{lemma}\label{efp}
$\mathcal P(\omega^{\vec\kappa})\geq \sum_{i=1}^k\kappa_iq(\bar{x}_i)+o(|\vec\kappa|).$
\end{lemma}
\begin{proof}
Since $\partial D$ is smooth, $D$ satisfies the interior sphere condition at $\bar x_i\in \partial D$ for $i=\{1,\cdot\cdot\cdot,k\}$. Therefore for $|\vec\kappa|$ sufficiently small we can choose a disc
$B_{\varepsilon_i}( x_i^{\vec\kappa})\subset D$ with $| x_i^{\vec\kappa}-\bar x_i|=\varepsilon_i$, where $\varepsilon_i$ satisfies $\pi\varepsilon_i^2=\sqrt{\kappa_i/\Lambda(\kappa_i)}$. Set $w=\sum_{i=1}^k \sqrt{\kappa_i\Lambda(\kappa_i)} I_{B_{\varepsilon_i}(x_i^{\vec\kappa})}.$ Then $w\in \mathcal{N}^{\vec\kappa}$ if $|\vec\kappa|$ is sufficiently small. Since $\omega^{\vec\kappa}$ is a maximizer, we have
\begin{equation}
  \mathcal P(\omega^{\vec\kappa})\geq \mathcal P(w).
  \end{equation}
On the other hand, using Lemma \ref{oo1} we can calculate directly as in Lemma \ref{lowerb} to obtain
\[\mathcal P(w)=\sum_{i=1}^k\kappa_iq(\bar{x}_i)+o(|\vec\kappa|).\]
Thus the proof is finished.

\end{proof}

\begin{lemma}
$\Lambda_i(\kappa_i)\int_DF({\Lambda_i^{-1}(\kappa_i)}{\omega_i^{\vec\kappa}})dx= o(|\vec\kappa|)$, where $\omega_i^{\vec\kappa}=\omega^{\vec\kappa}I_{B_{r_0}(\bar x_i)}$.
\end{lemma}
\begin{proof}
Recall that
\begin{equation}
\mathcal P(\omega^{\vec\kappa})=\frac{1}{2}\int_D\omega^{\vec\kappa}\mathcal G\omega^{\vec\kappa} dx+\int_Dq\omega^{\vec\kappa} dx
  -\sum_{i=1}^k\int_D\Lambda_i(\kappa_i)F_i(\Lambda_i^{-1}(\kappa_i)\omega_i^{\vec\kappa})dx.
\end{equation}
By Lemma \ref{efp} we deduce that
  \begin{align}
  \sum_{i=1}^k\int_D\Lambda_i(\kappa_i)F_i(\Lambda_i^{-1}(\kappa_i)\omega_i^{\vec\kappa})dx&\leq \frac{1}{2}\int_D\omega^{\vec\kappa}\mathcal G\omega^{\vec\kappa} dx+\int_Dq\omega^{\vec\kappa} dx-\sum_{i=1}^k\kappa_iq(\bar{x}_i)+o(|\vec\kappa|)\\
  &\leq o(|\vec\kappa|),
  \end{align}
from which the desired statement apparently follows.
\end{proof}

\begin{lemma}\label{efmu}
For each $i\in\{1,\cdot\cdot\cdot,k\},$ there holds
\begin{equation}\label{claim}
\int_D\omega_i^{\vec\kappa}(\mathcal{G}\omega^{\vec\kappa}+q-\mu_i^{\vec\kappa})dx=o(|\vec\kappa|).
\end{equation}
\end{lemma}
\begin{proof}
  Recall that $\phi^{\vec\kappa}_i=\mathcal{G}\omega^{\vec\kappa}+q-\mu_i^{\vec\kappa}.$
We calculate as follows
\begin{align*}
&\int_D\omega_i^{\vec\kappa}\phi_i^{\vec\kappa}dx\\
=&\int_{\{0<\omega_i^{\vec\kappa}<\Lambda_i(\kappa_i)\}}\omega_i^{\vec\kappa}\phi_i^{\vec\kappa}dx
+\int_{\{\omega_i^{\vec\kappa}=\Lambda_i(\kappa_i)\}}\omega_i^{\vec\kappa}\phi_i^{\vec\kappa}dx\\
=&\int_{\{0<\omega_i^{\vec\kappa}<\Lambda_i(\kappa_i)\}}\omega_i^{\vec\kappa} f_i^{-1}(\Lambda_i^{-1}(\kappa_i)\omega_i^{\vec\kappa})dx
+\int_{\{\omega_i^{\vec\kappa}=\Lambda_i(\kappa_i)\}}\omega_i^{\vec\kappa}\phi_i^{\vec\kappa}dx\\
=&\int_D\omega_i^{\vec\kappa} f_i^{-1}(\Lambda_i^{-1}(\kappa_i)\omega_i^{\vec\kappa})dx
+\int_{\{\omega_i^{\vec\kappa}=\Lambda_i(\kappa_i)\}}\omega_i^{\vec\kappa}(\phi_i^{\vec\kappa}-f_i^{-1}(1))dx\\
=&\int_D\omega_i^{\vec\kappa} f_i^{-1}(\Lambda_i^{-1}(\kappa_i)\omega_i^{\vec\kappa})dx
+\int_D\omega_i^{\vec\kappa}(\phi_i^{\vec\kappa}-f_i^{-1}(1))_+dx.
\end{align*}
For the first term, since $f_i$ satisfies (H3)$'$, we obtain
\begin{align*}
\int_D\omega_i^{\vec\kappa} f_i^{-1}(\Lambda_i^{-1}(\kappa_i)\omega_i^{\vec\kappa})dx&=\Lambda_i(\kappa_i)\int_D\Lambda_i^{-1}(\kappa_i)\omega_i^{\vec\kappa} f_i^{-1}(\Lambda_i^{-1}(\kappa_i)\omega_i^{\vec\kappa})dx\\
&\leq \frac{\Lambda_i(\kappa_i)}{\delta_1}\int_DF_i(\Lambda_i^{-1}(\kappa_i)\omega_i^{\vec\kappa})dx\\
&=o(|\vec\kappa|).
\end{align*}
Therefore, to prove \eqref{claim}, it remains to show that
\begin{align}\label{wantt}
\int_D\omega_i^{\vec\kappa}(\phi_i^{\vec\kappa}-f_i^{-1}(1))_+dx=o(|\vec\kappa|).
\end{align}
Denote $\zeta_i^{\vec\kappa}=\phi_i^{\vec\kappa}-f_i^{-1}(1)$ and $A_i^{\vec\kappa}=\{x\in D\mid \omega_i^{\vec\kappa}(x)=\Lambda_i(\kappa_i)\}.$
By H\"older's inequality and the Sobolev imbedding theorem we have
\begin{equation}\label{gz}
\begin{split}
&\int_D\omega_i^{\vec\kappa}(\phi_i^{\vec\kappa}-f_i^{-1}(1))_+dx\\
=&\Lambda_i(\kappa_i)\int_{A_i^{\vec\kappa}}(\zeta_i^{\vec\kappa})_+dx\\
\leq&\Lambda_i(\kappa_i)|A_i^{\vec\kappa}|^{1/2}\left(\int_D|(\zeta_i^{\vec\kappa})_+ |^2dx\right)^{1/2}\\
\leq&C\Lambda_i(\kappa_i)|A_i^{\vec\kappa}|^{1/2}\int_D\left(|(\zeta_i^{\vec\kappa})_+ |+|\nabla(\zeta_i^{\vec\kappa})_+|\right)dx\\
=&C\Lambda_i(\kappa_i)|A_i^{\vec\kappa}|^{1/2}\int_D|(\zeta_i^{\vec\kappa})_+ |dx
+C\Lambda_i(\kappa_i)|A_i^{\vec\kappa}|^{1/2}\int_D|(\nabla\zeta_i^{\vec\kappa})_+|dx.
\end{split}
\end{equation}
Here $C$ does not depend on $|\vec\kappa|$. Since $|A_i^{\vec\kappa}|\to0$ as $|\vec\kappa|\to0,$ we obtain from \eqref{gz}
\begin{align}\label{Lamb}
\Lambda_i(\kappa_i)\int_D|(\zeta_i^{\vec\kappa})_+ |dx
\leq C\Lambda_i(\kappa_i)|A_i^{\vec\kappa}|^{1/2}\int_D|\nabla(\zeta_i^{\vec\kappa})_+|dx.
\end{align}
But
\begin{align*}
\Lambda_i(\kappa_i)|A_i^{\vec\kappa}|^{1/2}\int_D|\nabla(\zeta_i^{\vec\kappa})_+|dx
&\leq \Lambda_i(\kappa_i)|A_i^{\vec\kappa}|\left(\int_D|\nabla(\zeta_i^{\vec\kappa})_+|^2dx\right)^{1/2}\\
&\leq C\kappa_i\left(\int_{A_i^{\vec\kappa}}|\nabla q|^2+|\nabla \mathcal G\omega^{\vec\kappa}|dx\right)^{1/2}\\
&\leq o(|\vec\kappa|),
\end{align*}
which together with \eqref{gz} and \eqref{Lamb} implies \eqref{wantt}. This completes the proof.

\end{proof}

Now we are ready to estimate the Lagrange multiplier $\mu_i^{\vec\kappa}$.
\begin{lemma}\label{muuu}
$\mu_i^{\vec\kappa}=q(\bar x_i)+o(1).$
\end{lemma}
\begin{proof}
By Lemma \ref{qleq}, it suffices to prove
\[\mu_i^{\vec\kappa}\geq q(\bar x_i)+o(1).\]
We write
\begin{align*}
&\mathcal{P}(\omega^{\vec\kappa})\\
=&\frac{1}{2}\int_D\omega^{\vec\kappa}\mathcal{G}\omega^{\vec\kappa}dx+\int_Dq\omega^{\vec\kappa}dx
-\sum_{i=1}^k\Lambda_i(\kappa_i)\int_DF_i({\Lambda_i^{-1}(\kappa_i)}\omega_i^{\vec\kappa})dx\\
=&-\frac{1}{2}\int_D\omega^{\vec\kappa}\mathcal{G}\omega^{\vec\kappa}dx
+\sum_{i=1}^k\int_D\omega_i^{\vec\kappa}\phi_i^{\vec\kappa}dx
-\sum_{i=1}^k\Lambda_i(\kappa_i)\int_DF_i({\Lambda_i^{-1}(\kappa_i)}\omega^{\vec\kappa}_i)dx+\sum_{i=1}^k\kappa_i\mu_i^{\vec\kappa}\\
=&\sum_{i=1}^k\kappa_i\mu_i^{\vec\kappa}+\sum_{i=1}^k\int_D\omega_i^{\vec\kappa}\phi_i^{\vec\kappa}dx+o(|\vec\kappa|)\\
=&\sum_{i=1}^k\kappa_i\mu_i^{\vec\kappa}+o(|\vec\kappa|).
\end{align*}
Here we used Lemma \ref{efmu}. Combining Lemma \ref{efp} we obtain
\begin{equation}\label{fin}
\sum_{i=1}^k\kappa_i\mu_i^{\vec\kappa}\geq \sum_{i=1}^k\kappa_iq(\bar{x}_i)+o(|\vec\kappa|).
\end{equation}
Finally by Lemma \ref{qleq} and \eqref{fin} we deduce that
\begin{align*}\mu_i^{\vec\kappa}&=\frac{1}{\kappa_i}\left(\sum_{j=1}^k\kappa_j\mu_j^{\vec\kappa}-\sum_{j=1,j\neq i}^k\kappa_j\mu_j^{\vec\kappa}\right)\\
&\geq\frac{1}{\kappa_i}\left(\sum_{i=1}^k\kappa_iq(\bar{x}_i)+o(|\vec\kappa|)-\sum_{j=1,j\neq i}^k\kappa_jq(\bar x_j)\right)\\
&= q(\bar x_i)+o(1).
\end{align*}
This completes the proof of the lemma.

\end{proof}

\begin{lemma}
$\{x\in D\mid\mathcal G\omega^{\vec\kappa}(x)+q(x)-\mu_i^{\vec\kappa}\geq f_i^{-1}(1)\}=\varnothing$ if $|\vec\kappa|$ is sufficiently small. As a consequence, $\omega^{\vec\kappa}$ has the following form
\[\omega^{\vec\kappa}=\sum_{i=1}^k\Lambda_i(\kappa_i) f_i(G\omega^{\vec\kappa}+q-\mu_i^{\vec\kappa})I_{B_{r_0}(\bar x_i)}\]
if $|\vec\kappa|$ is sufficiently small.
\end{lemma}
\begin{proof}
Observe that for fixed $i$, $f_i^{-1}(1)$ is a positive number not depending on $\vec\kappa$. Combining Lemma \ref{muuu} we can easily get the desired result.
\end{proof}

\begin{lemma}\label{shrink}
For each $i\in\{1,\cdot\cdot\cdot,k\}$, $\{x\in D\mid \mathcal G\omega^{\vec\kappa}(x)+q(x)-\mu_i^{\vec\kappa}>0\}$ shrinks $\bar x_i$ as $|\vec\kappa|\to0^+.$
\end{lemma}
\begin{proof}
  The proof is quite similar to that given in Lemma \ref{shrinkk} and therefore is omitted.
\end{proof}
\subsection{Proof of Theorem \ref{multiple}}
Now we can give the proof of Theorem \ref{multiple}.
\begin{proof}[Proof of Theorem \ref{multiple}]
We only need to show that $\omega^{\vec\kappa}$ is a weak solution to the vorticity equation if $|\vec\kappa|$ is sufficiently small.
This is similar to the proof of Theorem \ref{single}.

For any $\phi\in C_c^\infty(D)$, let $\Phi_s$ be determined by \eqref{ode}. Set $\omega_s(x):=\omega^{\vec\kappa}(\Phi_s(x)).$
Since $\Phi_s$ is area-preserving and smooth, taking into account Lemma \ref{shrink}, we get $\omega_s\in\mathcal N^{\vec\kappa}$ if $|\vec\kappa|$ is sufficiently small.
So we have
\begin{equation}\label{var00}
\frac{d\mathcal P(\omega_s)}{ds}\bigg|_{s=0}=0.
\end{equation}
Expanding $\mathcal P(\omega_s)$ at $s=0$ gives
\[\begin{split}
\mathcal P(\omega_s)
=\mathcal P(\omega^{\vec\kappa})+s\int_D\omega^{\vec\kappa}\nabla^\perp(\mathcal G\omega^{\vec\kappa}+q)\cdot\nabla\phi dx+o(s),
\end{split}\]
which together with \eqref{var00} gives the desired result.

\end{proof}


\phantom{s}
 \thispagestyle{empty}

\end{document}